\documentclass{amsart}

\newtheorem{theorem}{Theorem}[section]
\newtheorem{lemma}[theorem]{Lemma}
\newtheorem{proposition}[theorem]{Proposition}

\newtheorem{problem}[theorem]{Problem}

\theoremstyle{definition}
\newtheorem{definition}[theorem]{Definition}
\newtheorem{example}[theorem]{Example}

\newtheorem{remark}[theorem]{Remark}

\def\forch{{\mathrm{ch}}}
\def\fgmod{{\mathcal{C}}}

\usepackage{amscd,amssymb}
\begin{document}

\title[Rationality of Hilbert series in noncommutative invariant theory]
{Rationality of Hilbert series\\
in noncommutative invariant theory}

\author[M\'aty\'as Domokos and Vesselin Drensky]
{M\'aty\'as Domokos and Vesselin Drensky}
\address{MTA Alfr\'ed R\'enyi Institute of Mathematics,
Re\'altanoda utca 13-15, 1053 Budapest, Hungary}
\email{domokos.matyas@renyi.mta.hu}
\address{Institute of Mathematics and Informatics,
Bulgarian Academy of Sciences,
Acad. G. Bonchev Str., Block 8,
1113 Sofia, Bulgaria}
\email{drensky@math.bas.bg}

\thanks{This project was carried out in the framework of the exchange program
between the Hungarian and Bulgarian Academies of Sciences.
It was partially supported by the Hungarian National Research, Development and Innovation Office,  NKFIH K 119934
and by Grant I02/18 of the Bulgarian National Science Fund.}

\subjclass[2010]{16R10, 13A50, 13A02, 15A72, 16W22, 16W50, 05E10.}

\keywords{rational Hilbert series, Schur function, relatively free associative algebras,
rational representation, unipotent subgroup, reductive group, tensor algebra, noncommutative invariant theory}

\begin{abstract}
It is a fundamental result in commutative algebra and invariant theory that a finitely generated graded module
over a commutative finitely generated graded algebra has a rational Hilbert series, and consequently
the Hilbert series of the algebra of polynomial invariants of a group of linear transformations is rational,
whenever this algebra is finitely generated.
This basic principle is applied here to prove rationality of Hilbert series of algebras of invariants that
are neither commutative nor finitely generated. Our main focus  is on linear groups acting on certain factor algebras
of the tensor algebra that arise naturally in the theory of polynomial identities.
\end{abstract}

\maketitle

%\section{Introduction and background}

\section{Introduction}\label{sec:intro}

A natural possible framework to develop noncommutative invariant theory is to investigate the action of
a linear transformation group $G\leq GL(V)$  on the factor of the tensor algebra $K\langle V\rangle$
modulo a $GL(V)$-stable ideal $I$, where $V$ is a finite dimensional vector space over
a base field $K$ of characteristic 0. We shall assume that $\dim(V)\ge 2$.
A special class of $GL(V)$-stable ideals in $K\langle V\rangle$ are the T-ideals.
Recall that an ideal in $K\langle V \rangle$ is a {\it T-ideal} if it is stable  with respect to
all $K$-algebra endomorphisms of $K\langle V\rangle$.
T-ideals are associated with varieties of associative  $K$-algebras. The {\it variety} $\mathfrak{R}$ defined by a system
of elements $J=\{f_j\}$ in the free associative algebra $K\langle x_1,x_2,\ldots\rangle$ of countable rank
consists of all associative $K$-algebras $A$ such that  $f_j=0$ is a polynomial identity on $A$ for all $f_j\in J$.
Identifying the tensor algebra $K\langle V\rangle$
with the free associative algebra $K\langle x_1,\ldots,x_n\rangle$ of rank $n=\dim(V)$,
the T-ideal $I=I_n(\mathfrak{R})$ is recovered as the set of $n$-variable identities satisfied by all algebras in
a suitable variety $\mathfrak{R}$. Therefore in this case we shall write $K\langle V\rangle /I =F_n(\mathfrak{R})$,
and call it a {\it relatively free algebra}, since it is a free object  generated by $n$ elements in the variety
$\mathfrak{R}$.
For a background on T-ideals and polynomial identities of associative algebras see \cite{DF}.

We shall assume that $I_n(\mathfrak{R})$ is non-zero, so $F_n(\mathfrak{R})$ is not the free algebra.
In the special case when  $\mathfrak{R}$  is the variety of commutative algebras, we recover
$S(V)=K[x_1,\ldots,x_n]$, the symmetric tensor algebra of $V$, which can be thought of as the algebra of polynomial functions
on the dual space $V^*$ of $V$.
For surveys on the study of $F_n(\mathfrak{R})^G$ see \cite{F} and \cite{D3}.
Loosely speaking, the moral of these works is that one can expect results analogous to the commutative case only when
$F_n(\mathfrak{R})$ is not very far from the commutative polynomial algebra $K[x_1,\ldots,x_n]$. For example,
benchmark results of commutative invariant theory are the statements that the algebras of (commutative) polynomial invariants
of finite groups or more generally, reductive groups are finitely generated.   Kharchenko \cite{Kh} characterized
the class of varieties $\mathfrak{R}$ such that $F_n(\mathfrak{R})^G$ is finitely generated for all finite groups $G$,
and this class turned out to be  rather narrow.
The present authors in \cite{DD}
characterized the even narrower class of varieties such that  $F_n(\mathfrak{R})^G$ is finitely generated for all  reductive $G$.
Furthermore, we mention that the first named author  in \cite{Do} extended the Sheppard-Todd-Chevalley Theorem by showing
that $F_n(\mathfrak{R})^G\cong F_n(\mathfrak{R})$ holds for some finite group $G$ if and  only if $G$
is a pseudoreflection group and the T-ideal $I(\mathfrak{R})$ contains $[[x_1,x_2],x_2]$, where we write $[a,b]=ab-ba$.

In contrast with that, in the present paper we deal with a general fact from commutative invariant theory that remains valid
for all noncommutative relatively free algebras.
Namely, by the so-called Hilbert-Serre Theorem, whenever $K[x_1,\ldots,x_n]^G$ is finitely generated, its Hilbert series is rational.
Notably this holds when $G$ is reductive or $G$ is a maximal unipotent subgroup of a reductive subgroup of $GL(V)$.
Our main result Theorem~\ref{thm:main} is that the Hilbert series of the graded algebra $F_n(\mathfrak{R})^G$ is rational,
provided that $G\le GL(V)$ is a group 
for which the finite generation property holds for subalgebras of $G$-invariants in finitely generated   commutative graded algebras $R$ on which $GL(V)$ acts rationally 
(in particular, for a closed reductive subgroup $G$ of $GL(V)$ or for a maximal unipotent subgroup $G$ of a closed reductive subgroup of $GL(V)$). 
Consequently, although $F_n(\mathfrak{R})^G$ is rarely finitely generated even for reductive groups,
it always has rational Hilbert series for $G$ reductive or a maximal unipotent subgroup of a reductive subgroup of $GL_n$.

A key ingredient of our proof is that extending a result  of Belov \cite{Bel}, Berele \cite{Ber} proved that the multivariate
Hilbert series of a relatively free algebra is of special form (called {\it nice} by him). In Section~\ref{sec:nice_rational}
(working in a more general setup than \cite{Ber}), we give in
Proposition~\ref{prop:characterization_nice_rational} a characterization of nice rational symmetric functions in terms of formal
characters of the general linear group $GL_n$ and finitely generated modules over commutative polynomial algebras,
endowed with a compatible rational $GL_n$-action.
This characterization allows us to deduce in Section~\ref{sec:hilbert_series} the general
Theorem~\ref{thm:condition_a} about rationality of the Hilbert series of the subspace of $G$-invariants in a graded
vector space on which $GL_n$ acts rationally such that the formal character is a nice rational function, provided
that the subgroup $G\le GL_n$ satisfies a condition formulated  in terms of commutative invariant theory.
In Section~\ref{sec:rationality of Nagata} we discuss the problem for the rationality of the Hilbert series of non-finitely generated algebras
of commutative polynomial invariants. We compute the Hilbert series of a textbook variant of the famous example of Nagata of an algebra of commutative invariants which is not finitely generated.  It turns out  that this algebra still has a nice rational Hilbert series.
We also give an example of a normal graded subalgebra of $K[x_1,x_2]$ with transcendental Hilbert series, and formulate Problem~\ref{problem:strong_nagata} concerning rationality of Hilbert series in commutative invariant theory.
In Section~\ref{sec:applications} we apply Theorem~\ref{thm:condition_a} for relatively free algebras,
and deduce by the above mentioned result of Belov and Berele our main result Theorem~\ref{thm:main} on rationality
of the Hilbert series of the subalgebra of invariants $F_n(\mathfrak{R})^G$ for a notable class of groups  $G\le GL_n$.
In Section~\ref{sec:examples} we give examples which show that the condition that the multivariate
Hilbert series of the relatively free algebra is a nice rational function cannot be relaxed to the weaker condition that
it is a rational function. We also compute the Hilbert series of the algebra of invariants of the multiplicative group of the base field
acting on the 2-generated relatively free algebra in the variety generated by the $2\times 2$ matrix algebra.

Finally, we change our topic in Section~\ref{sec:transfer}, and consider group actions on PI-algebras
(algebras satisfying a polynomial identity), that are not necessarily relatively free algebras.
In fact a finitely generated associative $K$-algebra is a PI-algebra if and only if it is a homomorphic image
of some $F_n(\mathfrak{R})$ for some $n$
and a proper subvariety $\mathfrak{R}$ of the variety of associative algebras. Vonessen \cite{V} proved that if $G$
is a linearly reductive group acting rationally on a finitely generated
noetherian PI-algebra $R$, then the subalgebra $R^G$ is finitely generated and noetherian. We note here that the transfer
principle (a well-known tool in commutative invariant theory, see \cite{G}) can be applied in this noncommutative setting as well,
and assert in Theorem~\ref{thm:transfer} that under the  conditions in the theorem of Vonessen, $R^H$
is finitely generated and noetherian for any subgroup $H$ of $G$ for which $K[G]^H$ is finitely generated,
where the action of $H$ on the coordinate ring $K[G]$ of $G$ is induced by the right translation action of $H$ on $G$.

We close the introduction by mentioning that a more general setup for developing noncommutative invariant theory that also received attention is provided by Hopf algebra actions on noncommutative rings, see for example the survey 
\cite{bahturin}, the papers \cite{gr-hr}, \cite{ery}, and the references therein.  
A typical  problem considered in these works is whether the existence of a 
polynomial identity for the algebra of invariants implies the existence of an identity for the whole algebra.

%%%%%%%%%%%%%%%%%%%%%%%%%%%%%%%%

\section{A characterization of nice rational symmetric functions}\label{sec:nice_rational}

Write $GL_n$  for the general linear group $GL_n(K)$ over $K$, viewed as an affine algebraic group.
In this paper by a {\it representation of} $GL_n$ we shall always mean a {\it rational representation}, i.e.,
a $K$-vector space $Z$ together with a group homomorphism $\rho:GL_n\to GL(Z)$ giving an  action of $GL_n$ on
$Z$ via linear transformations such that $Z$ is spanned by finite dimensional $GL_n$-invariant subspaces
and for any finite dimensional $GL_n$-invariant subspace $W$ of $Z$, the map $GL_n\to GL(W)$, $g\mapsto \rho(g)\vert_W$
is a morphism of affine algebraic varieties. When $Z$ is finite dimensional, define
the {\it formal character} of $Z$ as the element $\forch_Z$ in the Laurent polynomial ring $\mathbb{Z}[t_1^{\pm 1},\ldots,t_n^{\pm 1}]$
satisfying that for any diagonal matrix
$\mathrm{diag}(z_1,\ldots,z_n)\in GL_n$ with diagonal entries $z_1,\ldots,z_n\in K\setminus \{0\}$ we have the equality
$\mathrm{Tr}(\rho(\mathrm{diag}(z_1,\ldots,z_n)))=\forch_V(z_1,\ldots,z_n)$.
In fact $\forch_V\in \mathbb{N}_0[t_1^{\pm 1},\ldots,t_n^{\pm 1}]$ has non-negative coefficients and is symmetric in the variables $t_1,\ldots,t_n$.
Hence it belongs to the subring
$\mathbb{Z}[t_1^{\pm 1},\ldots,t_n^{\pm 1}]^{S_n}$ of $S_n$-invariants in $\mathbb{Z}[t_1^{\pm 1},\ldots,t_n^{\pm 1}]$,
where the symmetric group $S_n$ acts by permuting the indeterminates.
It is well known that for any sequence $\lambda\in \mathbb{Z}^n$ with $\lambda_1\ge \cdots \ge \lambda_n$
there is a finite dimensional representation $L_{\lambda}$ of $GL_n$ whose formal character is given by the equality
\[
\forch_{L_{\lambda}}\cdot \prod_{1\le i<j\le n}(t_i-t_j)=\sum_{\sigma\in S_n}\mathrm{sign}(\sigma)\prod_{i=1}^nt_{\sigma(i)}^{\lambda_i+n-i}.
\]
The set of all $\forch_{L_{\lambda}}$ forms a basis of the free $\mathbb{Z}$-module $\mathbb{Z}[t_1^{\pm 1},\ldots,t_n^{\pm 1}]^{S_n}$.
Note that when $\lambda_n\ge 0$, the formal character $\forch_{L_{\lambda}}$ is the Schur function associated to the partition $\lambda$.
Moreover, all representations of $GL_n$ are completely reducible, and
$\{L_{\lambda}\mid \lambda\in \mathbb{Z}^n,\, \lambda_1\ge\cdots\ge\lambda_n\}$ is a complete irredundant list
of representatives of the isomorphism classes of irreducible representations of $GL_n$ (see for example Section 9.8.1 in \cite{P}).
We shall deal with the situation when $\displaystyle Z=\bigoplus_{d=0}^{\infty}Z_d$ is graded, each homogeneous component $Z_d$ is finite dimensional
and $GL_n$-invariant.
In this case we say that $Z$ is a {\it graded representation of} $GL_n$ and call the  {\it formal character of the
graded representation} $Z$ the formal power series
\[
\forch_Z(q)=\sum_{d=0}^\infty \forch_{Z_d} \cdot q^d\in \mathbb{Z}[t_1^{\pm 1},\ldots,t_n^{\pm 1}]^{S_n}[[q]].
\]

\begin{definition}\label{def:C}
Denote by $\fgmod$ the class of graded representations $Z$ of $GL_n$ that have also the structure of a finitely generated graded module
over the symmetric tensor algebra
(polynomial algebra) $S(W)$  for some finite dimensional graded representation $W$ of $GL_n$ (the grading on $W$ induces
a grading on $S(W)$ in the standard way),
such that the action of $GL_n$ on $Z$ and $S(W)$ is compatible with the $S(W)$-module structure of $V$ (i.e.,
for $g\in GL_n$, $f\in S(W)$, and $m\in Z$, we have
$g\cdot (fm)=(g\cdot f) (g\cdot m)$).
\end{definition}

\begin{proposition}\label{prop:characterization_nice_rational}
An element $f\in \mathbb{Z}[t_1^{\pm 1},\ldots,t_n^{\pm 1}]^{S_n}[[q]]$ belongs to the $\mathbb{Z}$-submodule
of $\mathbb{Z}[t_1^{\pm 1},\ldots,t_n^{\pm 1}]^{S_n}[[q]]$
generated by $\{\forch_V(q)\mid V\in \fgmod\}$ if and only if $f$ is of the form
\begin{equation}\label{eq:nice_rational}
f=\frac{P(t_1,\ldots,t_n,q)}{\displaystyle \prod_{(\alpha,k)}(1-t_1^{\alpha_1}\cdots t_n^{\alpha_n}q^{k})}
\end{equation}
where $P\in \mathbb{Z}[t_1^{\pm 1},\ldots,t_n^{\pm 1}][q]$ is a polynomial in $q$ and the product
in the denominator ranges over a finite multiset of pairs $(\alpha,k)$
where $\alpha\in \mathbb{Z}^n$, $k\in \mathbb{N}_0$.
\end{proposition}

\begin{proof}
 For $\lambda\in\mathbb{Z}^n$ with $\lambda_1\ge \cdots \ge \lambda_n$ and $k\in\mathbb{N}_0$ write
 \[
 e^{\lambda}_k=\prod(1-t_1^{\alpha_1}\cdots t_n^{\alpha_n}q^k)
 \]
 where a factor $(1-t_1^{\alpha_1}\cdots t_n^{\alpha_n}q^k)$ in the above product occurs exactly with multiplicity the coefficient of
 $t_1^{\alpha_1}\cdots t_n^{\alpha_n}$ in $\forch_{L_{\lambda}}\in \mathbb{N}_0[t_1^{\pm 1},\ldots,t_n^{\pm 1}]$.

Suppose that $f$ is of the form \eqref{eq:nice_rational}. For any $\alpha\in \mathbb{Z}^n$ and $k\in\mathbb{N}_0$,
$(1-t_1^{\alpha_1}\cdots t_n^{\alpha_n}q^{k})$ is a factor of $e^{\lambda}_{k}$, where $\lambda$ is obtained by
rearranging the elements in the sequence $(\alpha_1,\ldots,\alpha_n)$ in the
decreasing order. Therefore we may assume that the denominator of $f$ in \eqref{eq:nice_rational} is
$\displaystyle \prod_{(\lambda,k)\in A}e^{\lambda}_k$ for a finite multiset $A$ of pairs $(\lambda,k)$
where $k\in\mathbb{N}_0$, $\lambda\in \mathbb{Z}^n$ with
$\lambda_1\ge\cdots\ge\lambda_n$.
Then the denominator of $f$ is symmetric, hence the numerator
$P$ of $f$ belongs to
$\mathbb{Z}[t_1^{\pm 1},\ldots,t_n^{\pm 1}]^{S_n}[q]$ and can be written as an integral linear combination
\[
P=\sum_{(\mu,r)\in B}m_{\mu,r} \forch_{L_{\mu}}(t_1,\ldots,t_n)q^r
\]
where $B$ is a finite multiset of pairs $(\mu,r)$, $r\in \mathbb{N}_0$, $\mu\in \mathbb{Z}^n$,
$\mu_1\ge\cdots \ge \mu_n$.
Take  the finite dimensional graded $GL_n$-representation
\[
W=\bigoplus W_k, \quad
W_k=\bigoplus_{\{\lambda\mid (\lambda,k)\in A\}}L_{\lambda}.
\]
By construction of $W$ the graded $GL_n$-representation $S(W)$ has formal character
\[
\forch_{S(W)}(t)=\frac{1}{\displaystyle \prod_{(\lambda,k)\in A}e^{\lambda}_k}.
\]
For  $(\mu,r)\in B$ denote by $L_{\mu}^{[r]}$ the finite dimensional graded $GL_n$-representation whose degree $r$ component is $L_{\mu}$ and
all other homogeneous components are zero; the formal character of this graded representation is
$q^r\forch_{L_{\mu}}(t_1,\ldots,t_n)$.
Consider now the graded $GL_n$-representation $S(W)\otimes L_{\mu}^{[r]}$. It is naturally a graded  $S(W)$-module, it is free and is generated by the
finite dimensional subspace $1\otimes L_{\mu}^{[r]}$. Moreover, the $GL_n$-action on $S(W)$ is compatible with the action on
$S(W)\otimes L_{\mu}^{[r]}$.
Thus  $S(W)\otimes L_{\mu}^{[r]}$ belongs to $\fgmod$, and
\[
\forch_{S(W)\otimes L_{\mu}^{[r]}}(q)=\forch_{S(W)}(q)\cdot \forch_{L_{\mu}^{[r]}}(q)
=\frac{\forch_{L_{\mu}}(t_1,\ldots,t_n)q^r}{\displaystyle \prod_{(\lambda,k)\in A}e^{\lambda}_k}.
\]
Summarizing, we obtained
\[
f=\frac{\sum_{(\mu,r)\in B}m_{\mu,r} \forch_{L_{\mu}}(t_1,\ldots,t_n)q^r}{\displaystyle \prod_{(\lambda,k)\in A}e^{\lambda}_k}
=\sum_{(\mu,r)\in B}m_{\mu,r}\forch_{S(W)\otimes L_{\mu}^{[r]}}
\]
showing that $f$ belongs to $\displaystyle \sum_{Z\in \fgmod} \mathbb{Z}\forch_Z$.

To see the converse, take $Z\in\fgmod$. Then $Z$ is a finitely generated graded module over the commutative polynomial algebra $S(W)$,
where $\displaystyle W=\bigoplus_{k=0}^{\infty}W_k$ is a finite dimensional graded representation of $GL_n$. By the Hilbert Syzygy Theorem,
there exists a $d\le \dim(W)$ and an exact sequence
\begin{equation}\label{eq:exact} 0\to F_d\to F_{d-1} \to \cdots \to F_1\to F_0\to Z\to 0\end{equation}
of graded $S(W)$-module homomorphisms, where the $F_i$ are of the form $S(W)\otimes U_i$
for some finite dimensional graded representations $U_i$ of $GL_n$
(so they are free $S(W)$-modules belonging to $\fgmod$), and the maps are $GL_n$-equivariant.
Indeed, take for $U_0$ a $GL_n$-stable direct complement of
$S(W)_+Z$, where $S(W)_+$ stands for the sum of the positive degree homogeneous components of $S(W)$.
Since $S(W)_+Z$ is a graded subspace of $Z$, we may also assume that $U_0$ is spanned by homogeneous elements.
Then $U_0$ minimally generates $Z$ as an $S(W)$-module and is finite dimensional by the graded Nakayama Lemma
(see for example Lemma 3.5.1 in \cite{DK}). So we get an $S(W)$-module  surjection
$\varphi_0:S(W)\otimes U_0\twoheadrightarrow Z$, which is $GL_n$-equivariant by construction. Moreover, identifying
the subspace $1\otimes U_0$ of $S(W)\otimes U_0$ with $U_0\subset Z$, the free module $S(W)\otimes U_0$ becomes graded,
and the surjection $\varphi_0$ is a homomorphism of graded modules. Next repeat the same construction for the kernel
of the homomorphism $\varphi_0$ instead of $Z$, to come up with $U_1$ and
$\varphi_1:S(W)\otimes U_1\twoheadrightarrow \ker(\varphi_0)$. Continue in the same way. This process stops
in at most $\dim(W)$ steps by the Hilbert Syzygy Theorem
(see for example Section 1.3.2 in \cite{DK}), and we obtain the desired exact sequence \eqref{eq:exact}.
It follows that $\displaystyle \forch_M(q)=\sum_{j=0}^d (-1)^j\forch_{F_i}(q)$.
Note that  $\forch_{F_i}=\forch_{S(W)}(q)\cdot \forch_{U_i}(q)$. Here $\forch_{U_i}(q)\in\mathbb{Z}[t_1^{\pm 1},\ldots,t_n^{\pm 1}]^{S_n}[q]$
is a polynomial in $q$,
whereas $\displaystyle \forch_{S(W)}(q)=\prod_{(\lambda,k)}(e^{\lambda}_k)^{-m_{\lambda,k}}$, where $m_{\lambda,k}$ stands for the multiplicity of $L_{\lambda}$
as a summand in $W_k$.  It follows that
\[
\forch_Z(q)=\frac{\sum_{i=0}^d (-1)^i\forch_{U_i}(q)}{\displaystyle \prod_{(\lambda,k)}(e^{\lambda}_k)^{m_{\lambda,k}}}
\]
is indeed of the form \eqref{eq:nice_rational}.
\end{proof}

\begin{definition} \label{def:berele}
Following Berele \cite{Ber} (but changing a bit the setup considered by him) we call an element
of $\mathbb{Z}[t_1^{\pm 1},\ldots,t_n^{\pm 1}][[q]]$
a {\it nice rational  function} if it is of the form \eqref{eq:nice_rational}.
\end{definition}

\begin{remark}\label{remark:nice_rational} 
(i) With this terminology Proposition~\ref{prop:characterization_nice_rational} asserts 
that the abelian group of symmetric nice rational functions is generated by 
$\{\forch_V(q)\mid V\in \fgmod\}$. 

(ii) Nice rational functions not depending on the indeterminates $t_1,\dots,t_n$  are rational functions in $\mathbb{Z}[[q]]$ with denominators which are products of
binomials $1-q^k$. Recall that the Hilbert-Serre Theorem asserts that the Hilbert series of a finitely generated graded module over a finitely generated commutative graded $K$-algebra $R$ with $R_0=K$ is a nice rational function in $\mathbb{Z}[[q]]$. 

(iii) If $f(t_1,\dots,t_n)$ is a nice rational function in the sense of \cite{Ber}, then making the substitution $t_i\mapsto t_iq$ we get that 
$f(t_1q,\dots,t_nq)$ is a nice rational function in the sense of our 
Defininition~\ref{def:berele}. 
\end{remark}

%%%%%%%%%%%%%%%%%%%%%%%%%%%%%%%%%%%%%%%%%%%%%%

\section{Hilbert series of fixed point subspaces}\label{sec:hilbert_series}

Fix now a subgroup $G$ of $GL_n$.  We write $Z^G$ for the subspace of $G$-fixed points in a representation $Z$ of $GL_n$.
Define
\[
D:\mathbb{Z}[t_1^{\pm 1},\ldots,t_n^{\pm 1}]^{S_n} \to \mathbb{Z}
\]
as the abelian group homomorphism with $D(\forch_{L_{\lambda}})=\dim(L_{\lambda}^G)$.
This definition makes sense, since the $\forch_{L_{\lambda}}$ constitute a free $\mathbb{Z}$-module basis in
$\mathbb{Z}[t_1^{\pm 1},\ldots,t_n^{\pm 1}]^{S_n}$.
Furthermore, we keep the notation $D$ for the abelian group homomorphism
\[
D:\mathbb{Z}[t_1^{\pm 1},\ldots,t_n^{\pm 1}]^{S_n}[[q]] \to \mathbb{Z}[[q]]\text{ with }
D\left(\sum_{d=0}^\infty c_dq^d\right)=\sum_{d=0}^\infty D(c_d)q^d.
\]
Since every representation of $GL_n$ is completely reducible, for any finite dimensional $GL_n$-representation $W$ we have
$D(\forch_W)=\dim W^G$, and consequently
for a graded representation $\displaystyle Z=\bigoplus_{d=0}^\infty Z_d$ of $GL_n$ we have that
\begin{equation}\label{eq:hilbert-forch} D(\forch_Z(q))=\sum_{d=0}^\infty \dim(Z_d^G)q^d=H(Z^G,q)\end{equation}
is the {\it Hilbert series of the graded vector space} $Z^G$.

We need the following technical lemma (we provide a proof, since though the arguments are well known,
we did not find a reference where this is stated exactly in the form as below):
\begin{lemma}\label{lemma:equivalent}
The following conditions are equivalent for a subgroup $G$ of $GL_n$:
\begin{itemize}
\item[(i)] For any $Z\in \fgmod$, the subspace $Z^G$ is finitely generated as an $S(W)^G$-module (where $W$
is as in Definition~\ref{def:C} of the class $\fgmod$).
\item[(ii)] For any finitely generated graded commutative $K$-algebra $R$ on which $GL_n$ acts rationally via graded $K$-algebra automorphisms,
the subalgebra $R^G$ of $G$-invariants is a finitely generated $K$-algebra.
\end{itemize}
\end{lemma}
\begin{proof} Suppose (i) holds, and take a $K$-algebra $R$ as in (ii). Denote by $R_+$ the sum of the positive degree homogeneous components in $R$.
Then $R_+$ contains a finite dimensional  $GL_n$-submodule $W$ such that
$W$ is spanned by homogeneous elements, and $W$ generates $R$ as a $K$-algebra.
The identity map $W\to W$ extends to a surjection $S(W)\to R$ of graded $K$-algebras, and this surjection is $G$-equivariant.
So in this way $R_+$ becomes a graded $S(W)$-module in the class $\fgmod$. Take a finite $S(W)^G$-module generating system of $R^G_+$
(it exists by assumption).
This is easily seen to be a finite $K$-algebra generating system of $R^G$ by the graded Nakayama Lemma (see for example Lemma 3.5.1 in \cite{DK}).

Assume next that (ii) holds, and let $Z\in \fgmod$.  We repeat the proof of Theorem 16.9 in \cite{G}: make the direct sum
$R=S(W)\oplus Z$ a graded $K$-algebra by imposing the multiplication $(r,v)\cdot (s,w)=(rs,rw+sv)$ for $r,s\in S(W)$ and $v,w\in Z$.
By assumption the algebra $R^G=S(W)^G\oplus Z^G$ is finitely generated, implying in turn that the $S(W)^G$-module $Z^G$ is finitely generated.
\end{proof}

\begin{remark}\label{remark:zariskiclosure} 
The equivalent conditions (i) and (ii) of Lemma~\ref{lemma:equivalent} hold for a subgroup $G$ of $GL_n$ if and only if they hold for its Zariski closure $\bar{G}$ in $GL_n$. 
Indeed, for a given element $f$ of an algebra $R$ as in (ii), the stabilizer of $f$ in $GL_n$ is 
Zariski closed, since $GL_n$ acts rationally on $R$. Therefore we have $R^G=R^{\bar{G}}$.  
\end{remark} 

We mention two important and well-known sufficient conditions that imply condition (ii) (and hence condition (i)) of Lemma~\ref{lemma:equivalent}:

\begin{proposition}\label{prop:unipotent}
The conditions (i) and (ii) of Lemma~\ref{lemma:equivalent} hold for $G$ in each of the following two cases:
\begin{enumerate}
\item $G$ is a Zariski closed reductive subgroup of $GL_n$.
\item $G$ is the unipotent radical of a Borel subgroup of a Zariski closed reductive subgroup of $GL_n$.
\end{enumerate}
\end{proposition}
\begin{proof}  (1) is classical, see for example Theorem A (a) on page 3 of \cite{G} and the reference therein.
(2) is due to Had\v{z}iev \cite{H} and Grosshans (in arbitrary characteristic, see Theorem 9.4 in \cite{G}).
\end{proof}

\begin{theorem}\label{thm:condition_a}
Suppose that the subgroup $G$ of $GL_n$ satisfies condition (i) (and hence (ii)) of Lemma~\ref{lemma:equivalent}, and let $Z$
be a graded representation of $GL_n$ for which $\forch_Z(q)$ is a nice rational function 
in the sense of Definition~\ref{def:berele}. 
Then the Hilbert series $H(Z^G,q)\in \mathbb{Z}[[q]]$ of the subspace of $G$-invariants in $Z$ is a rational function
of the form
\[
H(Z^G,q)=\frac{P(q)}{\displaystyle \prod_{j=1}^m(1-q^{d_j})}
\]
for some  positive integers $m,d_1,\ldots,d_m$ and a polynomial
$P\in \mathbb{Z}[q]$.
\end{theorem}

\begin{proof} Since  $\forch_Z(q)$ is a nice rational function and it is also symmetric in $t_1,\dots,t_n$, by Proposition~\ref{prop:characterization_nice_rational}
there exists a positive integer $r$,
$Z_1,\ldots,Z_r \in \fgmod$ and integers $a_1,\ldots,a_r$ such that
$\displaystyle \forch_Z(q)=\sum_{i=1}^ra_i\forch_{Z_i}$.
By \eqref{eq:hilbert-forch} we have
\begin{equation} \label{eq:H-D}
H(Z^G,q)=D(\forch_Z(q))=\sum_{i=1}^ra_i D(\forch_{Z_i})=\sum_{i=1}^ra_iH(Z_i^G,q).\end{equation}
Here  $Z_i$ is a finitely generated graded $S(W_i)$-module (with compatible $GL_n$-action) for some finite dimensional
graded representation $W_i$ of $GL_n$,  $i=1,\ldots,r$.
Since conditions (i) and (ii) of Lemma~\ref{lemma:equivalent} hold for $G$, we have that  the algebra $S(W_i)^G$ is finitely generated,
and $Z_i^G$ is a finitely generated graded $S(W_i)^G$-module for $i=1,\ldots,r$.  By the Hilbert-Serre Theorem
(see for example Theorem 11.1 in \cite{AM}) the finitely generated $S(W_i)^G$-module $Z_i^G$ has a rational Hilbert series $H(Z_i^G,q)$
of the stated form, implying in turn by \eqref{eq:H-D} that $H(Z^G,q)$ is a rational function of the stated form (i.e. a nice rational function in $\mathbb{Z}[[q]]$).
\end{proof}

%%%%%%%%%%%%%%%%%%%%%%%%%%%%%%%%%%%%

\section{Examples in commutative algebra}\label{sec:rationality of Nagata}

We do not know any example of an algebra of commutative polynomial invariants with Hilbert series which is not a rational function.
In this section we compute the Hilbert series of a famous algebra of invariants which is not finitely generated but as it turns out, still has a nice rational Hilbert series.
The first example of a finite dimensional $G$-module $V$ such that the corresponding algebra $K[V]^G=K[x_1,\ldots,x_n]^G$ of invariants
is not finitely generated was found by
Nagata (see \cite{nagata}). A simplification of Nagata's argument (resulting in an example with a group of smaller dimension)
was given by Steinberg \cite{steinberg}.
Here we shall compute the Hilbert series of this latter algebra of invariants.
We refer to the  books \cite{dolgachev}, \cite{mukai} for facts about this algebra of invariants.

Let $C$ be an irreducible cubic algebraic curve in the complex projective plane $\mathbb{P}^2$. The set of non-singular points in $C$ has a well-known
abelian group structure.
Choose  $9$ distinct points $(a_{i1}:a_{i2}:a_{i3})$ in this abelian group such that their sum is not a torsion element
and the matrix $(a_{ij})_{i=1,\ldots,9}^{j=1,2,3}$ has full rank $3$.
Let $G_1$ be the subgroup of $\mathbb{C}^9$ (the direct sum of $9$ copies of the additive group of $\mathbb{C}$) given by
\[
G_1=\{s=(s_1,\ldots,s_9)\mid \sum_{i=1}^9s_ia_{ij}=0,\quad j=1,2,3\}\subset \mathbb{C}^9
\]
so $G\cong \mathbb{C}^6$. Consider also the torus
\[
T=\{t=(t_1,\ldots,t_9)\mid \prod_{i=1}^9t_i=1\}\subset   (\mathbb{C}^\times)^9
\]
so $T\cong (\mathbb{C}^\times)^8$.
The group $G=G_1\times T\subset \mathbb{C}^9\times (\mathbb{C}^\times)^9$ acts on the $18$-variable polynomial algebra
$S=\mathbb{C}[x_1,\ldots,x_9,y_1,\ldots,y_9]$ via
\[
(s,t)\cdot x_i=t_ix_i,\quad (s,t)\cdot y_i=t_i(s_ix_i+y_i),\quad i=1,\ldots,9.
\]
Set
\[
A_j=\sum_{i=1}^9a_{ij}\frac{y_i}{x_i},\quad j=1,2,3, \quad D=x_1\cdots x_9.
\]
Then $A_1,A_2,A_3,D$ generate the subfield of $G$-invariants in the field of fractions of $S$, and in fact
\[
R=S^G=S\cap \mathbb{C}[A_1,A_2,A_3,D,D^{-1}].
\]
Thus $R$ has a bigrading $R=\displaystyle \bigoplus_{d\in\mathbb{N}_0,m\in\mathbb{Z}}R_{(d,m)}$, where
\[
R_{(d,m)}=S\cap \{D^{d-m}f(A_1,A_2,A_3)\mid f\in\mathbb{C}[x,y,z] \text{ is homogeneous of degree }d\}.
\]

\begin{theorem}\label{thm:nagata}
%\begin{itemize}
%\item[(i)]
{\rm (i) (Proposition 2.49 in \cite{mukai} or page 60 in \cite{dolgachev})}
The $\mathbb{C}$-algebra $R=S^G$ is not finitely generated.

%\item[(ii)]

{\rm (ii) (Lemma 2.50 in \cite{mukai} and Lemma 4.5 in \cite{dolgachev})}
For $d\in \mathbb{N}_0$ and $m\in \mathbb{Z}$ we have
\[
\dim R_{(d,m)}=\begin{cases} \displaystyle\binom{d+2}{2}, & \text{ if }\quad m\le 0;\\
\displaystyle\binom{d+2}{2}-9\binom{m+1}{2}, &\text{ if }\quad d\ge 3m>0;\\
0, &\text{ if }\quad 3m>d.\end{cases}
\]
%\end{itemize}
\end{theorem}
Observe that $R_{(d,m)}$ is contained in the homogeneous component of degree $9(d-m)$ of the polynomial algebra of $S$,
hence we have
\[
H(S,q)=\sum_{n=0}^\infty c_nq^{9n} \text{ where } c_n=\sum_{m\ge -n}\dim R_{(m+n,m)}.
\]
We infer from Theorem~\ref{thm:nagata} (ii)
that
\[
c_n=\sum_{m=1}^{\lfloor \frac{n}{2}\rfloor}\left(\binom{n+m+2}{2}-9\binom{m+1}{2}\right)+\sum_{m=-n}^0\binom{n+m+2}{2}
\]
where $\lfloor \frac{n}{2}\rfloor$ stands for the lower integer part of $\frac{n}{2}$.
Now we have
\[\sum_{m=-n}^0\binom{n+m+2}{2}=\binom{n+3}{3},\]
for $n=2s>0$ even we have
\[\sum_{m=1}^s\dim R_{(2s+m,m)}=\frac{1}{6}(10s^3+3s^2-7s)=10\binom{s}{3}+11\binom{s}{2}+s,\]
and for $n=2s+1$ odd we have
\[\sum_{m=1}^s\dim R_{(2s+1+m,m)}=\frac{1}{6}(10s^3+18s^2+8s)=10\binom{s}{3}+16\binom{s}{2}+6s.\]
We obtain
\begin{eqnarray*} \sum_{n=0}^\infty c_nq^n &=&\sum_{n=0}^\infty\binom{n+3}{3}q^n
\\ &+&\sum_{s=1}^\infty(10\binom{s}{3}+11\binom{s}{2}+s)q^{2s}
\\ &+& \sum_{s=0}^\infty (10\binom{s}{3}+16\binom{s}{2}+6s)q^{2s+1}
\\ &=& \frac{1}{(1-q)^4}+\frac{10q^6+10q^7}{(1-q^2)^4}+\frac{11q^4+16q^5}{(1-q^2)^3}+\frac{q^2+6q^3}{(1-q^2)^2}
\\ &=& \frac{1+4q+7q^2+10q^3+10q^4+4q^5}{(1-q^2)^4}.
\end{eqnarray*}
Summarizing, we found the following:

\begin{proposition}\label{prop:nagata-hilbert}
We have the equality
\[H(S^G,q)=\frac{1+4q^9+7q^{18}+10q^{27}+10q^{36}+4q^{45}}{(1-q^{18})^4}.\]
In particular, $H(S^G,q)$ is a nice rational function.
\end{proposition}

\begin{problem} \label{problem:strong_nagata}
Does there exist a $G$-module $V$ such that $K[V]^G=K[x_1,\ldots,x_n]^G$ has a non-rational Hilbert series?
\end{problem}

There are many examples of subalgebras of the polynomial algebra in several variables which are not finitely generated.
See e.g., \cite{nathanson} for a family of subalgebras between $K[x_1]$ and $K[x_1,x_2]$ which are not finitely generated as
$K[x_1]$-algebras. It is also easy to construct a graded subalgebra of $K[x_1,x_2]$ with transcendental Hilbert series.
One of the few general abstract properties that all subalgebras of invariants in the polynomial algebra
do have is that they are normal (integrally closed in their field of
fractions). Below we give an example showing that normality does not imply
rationality of the Hilbert series.

\begin{example}
Fix a positive irrational real number $\alpha$ and  consider the subalgebra
$A=\mathrm{Span}_K\{x_1^ix_2^j\mid j<\alpha i\}$ of $K[x_1,x_2]$.
The algebra $A$ is not finitely generated (see for example \cite{nathanson}).
The number of degree $d\ge 1$ monomials in $A$ is $\vert\{i\mid i\le d,\quad  d-i<\alpha i\}\vert$, which is the number of integers in the interval
$\displaystyle (\frac{d}{1+\alpha},d]$.
Set $\displaystyle \beta=1-\frac{1}{1+\alpha}=\frac{\alpha}{\alpha+1}$.
The Hilbert series of $A$ is
\[
H(A,q)=1+\sum_{d=1}^\infty(1+\lfloor d\beta \rfloor  )q^d=\frac{1}{1-q}+\sum_{d=1}^\infty \lfloor d\beta \rfloor  q^d,
\]
where $\lfloor d\beta \rfloor  $ stands for the lower integer part of $d\beta$.  Suppose that this power series is algebraic over $\mathbb{Z}(q)$. Then
$\displaystyle H_1(q)=\sum_{d=1}^\infty \lfloor d\beta \rfloor  q^d$ is algebraic as well, implying in turn that
$\displaystyle H_2(q)=(1-q)H_1(q)=\sum_{d=1}^\infty (\lfloor d\beta \rfloor  -\lfloor (d-1)\beta\rfloor)q^d$
is algebraic.
Since $0<\beta<1$, we have $c_d=\lfloor d\beta \rfloor  -\lfloor (d-1)\beta\rfloor \in \{0,1\}$, so the coefficients of $H_2(q)$ take only two integer values.
By the classical theorem of Fatou \cite{Fa}, see Borwein and Coons \cite{BC} for an easy proof, $H_2(q)$ must be a rational function.
As in the proof of Theorem 2 in \cite{BC}, this implies that the sequence $(c_d\mid d=1,2,\ldots)$ is eventually periodic.
This contradicts the assumption that $\beta$ is irrational. Indeed, suppose that there exist positive integers $N,p$ such that $c_{d+p}=c_d$ for all $d\ge N$.
Set $s=\vert\{1\le i \le p\mid c_{N+i}=1\}\vert$. Then for any positive integer $k$ we have
\[
\lfloor (N+kp)\beta \rfloor-\lfloor N\beta\rfloor =\sum_{i=N+1}^{N+kp}c_i=ks.
\]
We conclude that $\vert kp\beta-ks\vert\le 1$, and hence
$\displaystyle \vert\beta-\frac{s}{p}\vert\le \frac{1}{kp}$
holds for all $k$. This leads to the contradiction $\displaystyle \beta=\frac{s}{p}$.

We claim that $A$ is normal. Indeed, note that $A$ is the semigroup algebra of the subsemigroup $S=\{(i,j)\mid j\le \alpha i\}$
of the additive semigroup $\mathbb{N}_0^2$. The semigroup $S$ is saturated in $\mathbb{N}_0^2$, i.e.,
if $ms\in S$ for some $s\in \mathbb{N}_0^2$ and $m\in \mathbb{N}$ then $s\in S$.
Suppose that some element $f$ from the field of fractions of $A$ is integral over $A$.
Then  there exists a finitely generated subsemigroup $S'\subset S$ such that $f$ is integral over the subalgebra
$K[S']=\mathrm{Span}_K\{x_1^ix_2^j\mid (i,j)\in S'\}$ of $A$ (take all monomials that occur in the coefficients
of the monic polynomial in $A[x]$ satisfied by $f$). So $f$ belongs to the normalization of $K[S']$, which is $K[S'_{\mathrm{sat}}]$, where
$S'_{\mathrm{sat}}=\mathbb{Q}_{\ge 0}S'\cap \mathbb{N}_0^2$ is the saturation of $S'$ (see for example Proposition 7.25 in \cite{miller-sturmfels}).
Now since $S$ is saturated, we have
$S'_{\mathrm{sat}}\subset S$, so $f\in K[S]=A$, showing that $A$ is integrally closed in its field of fractions.
\end{example}

%%%%%%%%%%%%%%%%%%%%%%%%%%%%%%%%%%%%

\section{Applications in noncommutative invariant theory}\label{sec:applications}

Let $\mathfrak{R}$ be a proper subvariety of the variety of all associative $K$-algebras.
Denote by $F_n(\mathfrak{R})$ the relatively free algebra of rank $n$  in the variety $\mathfrak{R}$.
Then $F_n(\mathfrak{R})=K\langle V\rangle/I$ is the factor algebra of the tensor algebra  $K\langle V\rangle$ of an  $n$-dimensional vector space $V$
modulo a non-zero T-ideal $I$. Choosing a basis $\{x_1,\ldots,x_n\}$ in $V$ we identify $GL_n$ with $GL(V)$
and the tensor algebra $K\langle V\rangle$ with
the free associative algebra $K\langle x_1,\ldots,x_n\rangle$. This is graded as usual, so the generators $x_i$ have degree $1$.
Since $I$ is necessarily homogeneous,
there is an induced grading on $F_n(\mathfrak{R})$,  and  $\displaystyle F_n(\mathfrak{R})=\bigoplus_{d=0}^\infty F_n(\mathfrak{R})_d$ is a graded
representation of $GL_n$ in the sense of Section~\ref{sec:nice_rational}. For any subgroup $G\le GL_n$, the subalgebra
$\displaystyle F_n(\mathfrak{R})^G=\bigoplus_{d=0}^\infty F_n(\mathfrak{R})^G_d$ is spanned by homogeneous elements, so it is also graded.
Its Hilbert series is defined as
\[
H(F_n(\mathfrak{R})^G,q)=\sum_{d=0}^\infty \dim(F_n(\mathfrak{R})^G_d)q^d\in \mathbb{Z}[[q]].
\]

\begin{theorem}\label{thm:main}
Suppose that condition (i) (and hence (ii)) of Lemma~\ref{lemma:equivalent} holds for the subgroup $G$ of $GL_n$.
Then the Hilbert series $H(F_n(\mathfrak{R})^G,q)$ of the subalgebra of $G$-invariants in $F_n(\mathfrak{R})$ is rational of the form
\[
H(F_n(\mathfrak{R})^G,q)=\frac{P(q)}{\displaystyle \prod_{j=1}^m(1-q^{d_j})}
\]
for some  positive integers $m,d_1,\ldots,d_m$ and a polynomial
$P\in \mathbb{Z}[q]$.
\end{theorem}
\begin{proof}  To simplify notation set  $R=F_n(\mathfrak{R})$.
An  $\mathbb{N}_0^n$-grading on the free associative algebra $K\langle x_1,\ldots,x_n\rangle$
is defined by setting the multidegree of $x_i$ the $i$th standard basis vector $(0,\ldots,0,1,0,\ldots,0)$ in $\mathbb{Z}^n$.
This induces a multigrading $\displaystyle R=\bigoplus_{\alpha\in\mathbb{N}_0^n}R_{\alpha}$, and the {\it multivariate Hilbert series}
of $R$ is
\[
H(R,t_1,\ldots,t_n)=\sum_{\alpha\in\mathbb{N}_0^n}\dim(R_{\alpha})t_1^{\alpha_1}\cdots t_n^{\alpha_n}.
\]
Keep the notation $x_i$ for the image of $x_i$ under the natural surjection $K\langle x_1,\ldots,x_n\rangle \to R$.
Then there is a subset $S\subset \{1,\ldots,n\}^d$ such that the monomials $\{x_{i_1}\cdots x_{i_d}\mid (i_1,\ldots,i_d)\in S\}$
constitute  a basis of $R_{\alpha}$. Note that this basis of $R_{\alpha}$ consists of joint eigenvectors of the subgroup of
diagonal matrices in $GL_n$, since we have
\[
\mathrm{diag}(z_1,\ldots,z_n)\cdot x_i=z_ix_i,
\]
and consequently
\[
\mathrm{diag}(z_1,\ldots,z_n)\cdot x_{i_1}\cdots x_{i_d}=z_1^{\alpha_1}\cdots z_n^{\alpha_n}x_{i_1}\cdots x_{i_d}.
\]
For a fixed $d\in \mathbb{N}_0$  the degree $d$ homogeneous component  $R_d$ of
$R$ is $GL_n$-stable, and by the above considerations the sum of the terms  of $H(R,t_1,\ldots,t_n)$ of total degree $d$
is the formal character $\forch_{R_d}$ of $R_d$.
It follows that substituting $t_i\mapsto t_iq$ in the multivariate Hilbert series of $R$ we obtain
\[
H(R,qt_1,\ldots,qt_n)=\forch_{R}(q),
\]
the formal character of the graded $GL_n$-representation $R$.
Belov \cite{Bel} proved that the univariate Hilbert series $H(R,t)$ of the graded algebra $R$ is rational.
Berele strengthened this statement  in Theorem 1 of \cite{Ber}  by showing that the multigraded Hilbert series
$H(R,t_1,\ldots,t_n)$ of $R$ is nice (in the sense of \cite{Ber}), which implies that   $H(R,qt_1,\ldots,qt_n)$ is a   nice rational function in the
sense of Definition~\ref{def:berele}.
Summarizing, the formal character $\forch_{R}(q)$ is a nice rational function in the
sense of Definition~\ref{def:berele}, thus the statement of our theorem now follows from Theorem~\ref{thm:condition_a}.
\end{proof}

\begin{remark}\label{remark:rational_nonfinite} In contrast with  commutative invariant theory, the algebra $F_n(\mathfrak{R})^G$
is typically not finitely generated, even if
condition (i) and (ii) of Lemma~\ref{lemma:equivalent} hold for $G$. For example, by the proof of Theorem 3.1 in \cite{DD},
if $\mathfrak{R}$ is not a subvariety of a variety of Lie nilpotent algebras with a given index of Lie nilpotency,
then for any $n\ge 2$ there is a Zariski closed subgroup $G$  of $GL_n$ isomorphic to the multiplicative group of $K$ such that
$F_n(\mathfrak{R})^G$ is not finitely generated.
Therefore  Theorem~\ref{thm:main} yields interesting examples of naturally occurring
graded algebras, that are not finitely generated,
but have rational Hilbert series like
finitely generated commutative graded algebras.
\end{remark}

%%%%%%%%%%%%%%%%%%%%%%%%%%%%%%%%%%%%%

\section{Examples}\label{sec:examples}

\begin{example}\label{example:adf}
We provide two examples showing that in Theorem~\ref{thm:condition_a} the condition that $\forch_Z(q)$ is a \emph{nice}
rational function can not be relaxed to the weaker condition that
$\forch_Z(q)$ is a rational function.

(i) Let $V$ be $2$-dimensional and $G=SL(V)$, the special linear subgroup of $GL(V)$. In this case $G$
is a Zariski closed reductive subgroup of $GL(V)$. The tensor algebra $K\langle V\rangle$ has rational Hilbert series
\[H(K\langle V\rangle, t_1,t_2)=\frac{1}{1-t_1-t_2}, \]
whereas by Example 5.10 in \cite{ADF} we have that the subalgebra of $SL(V)$-invariants has non-rational Hilbert series
\[H(K\langle V\rangle^{SL(V)},q)=\frac{1}{2q^2}(1-\sqrt{1-4q^2})
=\sum_{n=0}^{\infty}\frac{1}{n+1}\binom{2n}{n}q^n.\]
Further examples of subalgebras of $SL_2$-invariants in free algebras $K\langle x_1,\ldots,x_n\rangle$
having non-rational Hilbert series can be found in \cite{ADF}.

(ii) Again let $V=\mathrm{Span}_K\{x,y\}$ be $2$-dimensional, and take
\[G=\{x\mapsto x,\   y\mapsto cx+y\mid c\in K\}\subset GL(V).\]
In this case $G$ is the unipotent radical of a Borel subgroup of $GL(V)$, and in fact $G$ is isomorphic to  the additive group of $K$.
By Proposition 5.3 in  \cite{DG}
we have that
\[
H((K\langle V\rangle)^G,q)=\sum_{p=0}^{\infty}
\left(\binom{2p}{p}q^{2p}+\binom{2p+1}{p}q^{2p+1}\right)
\]
\[
=\frac{1-\sqrt{1-4q^2}}{2q^2}\cdot
\frac{1}{\displaystyle 1-\frac{1-\sqrt{1-4q^2}}{2q}}.
\]
In particular, $H(K\langle V\rangle^G,q)$ is not rational.
\end{example}

\begin{example}\label{example:generic_two_by_two}
Here we compute the Hilbert series of a particular algebra of invariants  in a relatively free algebra along the lines of the proof of Theorem~\ref{thm:main}.
Let $V=\mathrm{Span}_K\{x,y\}$ be $2$-dimensional, let $I$ be the T-ideal in $K\langle V\rangle$ of $2$-variable identities of the algebra $K^{2\times 2}$ of $2\times 2$ matrices over
$K$, and $G=K^\times $, the multiplicative group of the base field acting on $V$ via
$z\cdot (ax+by)=zax+z^{-1}by$ ($z\in K^\times$, $a,b\in K$).
The Hilbert series of $R=K\langle V\rangle/I$ was computed in \cite{FHL}:
\begin{equation}\label{eq:fhl}
H(R,t_1,t_2)=\frac{1}{(1-t_1)(1-t_2)}+\frac{t_1t_2}{(1-t_1)^2(1-t_2)^2(1-t_1t_2)}.
\end{equation}
Taking into account that $\forch_{R}(q)=H(R,t_1q,t_2q)$ we deduce from \eqref{eq:fhl} the equality
\begin{equation} \label{eq:Z}
\forch_{R}(q)=\forch_Z(q)
\end{equation}
where $Z$ is the polynomial $GL(V)$-module
\[
Z=S(V)\oplus (\bigwedge^2(V)\otimes S(V\oplus V)\otimes S(\bigwedge^2(V))).
\]
Here $\displaystyle \bigwedge^2(V)$ is the exterior tensor square of $V$, whereas $S(V)$ and  $S(V\oplus V)$ are the symmetric tensor algebras   of $V$
and $V\oplus V$. The grading is given by the decomposition into isotypic components with respect to the action of the subgroup of scalar transformations in $GL(V)$.
Note that $G$ acts trivially on $\displaystyle \bigwedge^2(V)$, hence
\[
Z^G=S(V)^G\oplus (\bigwedge^2(V)\otimes S(V\oplus V)^G\otimes S(\bigwedge^2(V))).
\]
Clearly, $S(V)^G=K[xy]$ is a polynomial algebra generated by a degree $2$ element, and
$S(V\oplus V)^G=K[x_1,y_1,x_2,y_2]^G=K[x_1y_1,x_1y_2,x_2y_1,x_2y_2]$.
It is well known that the ideal of relations between the generators $x_iy_j$, $1\le i,j\le 2$, of $S(V\oplus V)^G$
is a principal ideal generated by the relation
$(x_1y_1)(x_2y_2)=(x_1y_2)(x_2y_1)$. It follows that
$\displaystyle H(S(V\oplus V)^G,q)=\frac{1-q^4}{(1-q^2)^4}$.
By \eqref{eq:hilbert-forch} and \eqref{eq:Z} we obtain
\[
H(R^G,q)=D(\forch_{R}(q))=
D(\forch_Z(q))=H(Z^G,q)=\frac{1}{1-q^2}+\frac{q^2(1+q^2)}{(1-q^2)^4}.
\]
Note that $R^G$ is not a finitely generated $K$-algebra. Indeed, $I$ is contained in $C^2$, the square of the commutator ideal of $K\langle V\rangle$,
since $C^2$ is known to be the T-ideal of $2$-variable identities of the subalgebra of upper triangular matrices in $K^{2\times 2}$
(see \cite{Ma}).
Thus $K\langle V\rangle/C^2$ is a homomorphic image of $R$, and hence $(K\langle V\rangle/C^2)^G$
is a homomorphic image of $R^G$ (recall that $G$ is linearly reductive). Moreover, using an example from \cite{Kh}
it was pointed out in the proof of Theorem 3.1 in \cite{DD} that
$(K\langle V\rangle/C^2)^G$ is not finitely generated, implying in turn that $R^G$ is not finitely generated.
\end{example}

\begin{example}
In some cases by a  non-commutative variant of the Molien-Weyl theorem the operator $D$ from Section~\ref{sec:hilbert_series}
can be evaluated using the Weyl integration formula and residue calculus.
An instance where the Hilbert series of the subalgebra of invariants in a relatively free algebra was computed
that way occurs in  \cite{Do2}.
Namely,
take for $V$ the space of $2\times 2$ matrices over $K$, on which $G=GL_2$ acts by conjugation, and let
$I$ be the T-ideal in $K\langle V\rangle=K\langle x_1,x_2,x_3,x_4\rangle$ generated by $[[x_1,x_2],x_2]$.
It was shown in Lemma 2.1 of \cite{Do2} that
\[
H((K\langle V\rangle/I)^{GL_2},t)=\frac{1+2t^3+t^4}{(1-t)(1-t^2)}.
\]
In that case the algebra $K\langle V\rangle /I$ is Noetherian and hence  $(K\langle V\rangle/I)^{GL_2}$ is finitely generated  (see \cite{DD}).
Explicit generators are also found in \cite{Do2}.
\end{example}

%%%%%%%%%%%%%%%%%%%%%%%%%%%%%%%

\section{Transfer principle}\label{sec:transfer}

In this section the characteristic of the base field $K$ can be positive, but we assume that $K$ is algebraically closed.
Let $G$ be a linear algebraic group with coordinate ring $K[G]$, and let $H$ be a  subgroup of $G$. Let $R$ be an associative $K$-algebra
on which $G$ acts rationally via $K$-algebra automorphisms. Then $G\times H$ acts  on $K[G]\otimes R$ via $K$-algebra automorphisms as follows:
\[
(g,h)\cdot \sum_i f_i \otimes r_i=\sum_i f_i(g^{-1}xh)\otimes gr_i\text{ for }g\in G,h\in H,\ f_i\in K[G], \ r_i\in R,
\]
where we write $f_i(g^{-1}xh)\in K[G]$ for the function $G\to K$, $x\mapsto f_i(g^{-1}xh)$.
Using the embeddings $G\to G\times H$, $g\mapsto (g,1)$ and $H\to G\times H$, $h\mapsto (1,h)$,  the $G\times H$-action on $K[G]\otimes R$
restricts to commuting actions of $G$ and $H$ on $K[G]\otimes R$. Recall that the subalgebra $(K[G]\otimes R)^G$ of $G$-invariants
is stable under the $H$-action, and similarly, the subalgebra
$(K[G]\otimes R)^H=K[G]^H\otimes R$ of $H$-invariants is $G$-stable.
Clearly we have
\[
(K[G]^H\otimes R)^G=((K[G]\otimes R)^H)^G=(K[G]\otimes R)^{G\times H}=((K[G]\otimes R)^G)^H.
\]

\begin{lemma}\label{lemma:transfer}
Denote by $\Phi:(K[G]\otimes R)^G\to R$  the restriction to $(K[G]\otimes R)^G$ of the linear map $K[G]\otimes R\to R$ given
by $f\otimes r\mapsto f(1_G)r$. Then we have the following:
\begin{itemize}
\item[(i)] $\Phi$ is an $H$-equivariant $K$-algebra isomorphism.
\item[(ii)] $\Phi$ restricts to a $K$-algebra isomorphism $(K[G]^H\otimes R)^G\to R^H$.
\end{itemize}
\end{lemma}

\begin{proof} This is  Theorem 9.1 in \cite{G} (in \cite{G} $R$ is assumed to be a commutative $K$-algebra,
but commutativity is nowhere used in the proof, which works verbatim in the present generality).
\end{proof}

\begin{theorem}\label{thm:transfer} Let $R$ be a  left noetherian finitely generated $K$-algebra satisfying
a polynomial identity, and $G$ a linearly reductive linear algebraic group acting rationally on $R$ via $K$-algebra automorphisms.
Suppose that $H$ is a subgroup of $G$ such that $K[G]^H$ is a finitely generated $K$-algebra (for example, $H$
can be a maximal unipotent subgroup of $G$). Then the subalgebra $R^H$ of $H$-invariants in $R$ is finitely generated and is left noetherian.
\end{theorem}

\begin{proof}
Recall that the polynomial algebra $S[x]$ over a left noetherian ring $S$ is left noetherian (see for example \cite{MR}).
An iterated use of this yields that
since $K[G]^H$ is a finitely generated commutative $K$-algebra and $R$ is left noetherian, the algebra  $K[G]^H\otimes R$ is left noetherian.
As both $K[G]^H$ and $R$ are finitely generated $K$-algebras, the same holds for $K[G]^H\otimes R$.
Moreover, $K[G]^H\otimes R$ is a PI-algebra, since it satisfies the same multilinear polynomial identities as $R$,
and every PI-algebra over any field $K$ necessarily satisfies a multilinear polynomial identity.
A theorem of Vonessen \cite{V} asserts that the subalgebra of invariants of the linearly reductive group $G$ acting rationally
on an affine left noetherian PI-algebra is finitely generated and left noetherian. By this theorem we conclude
that $(K[G]^H\otimes R)^G$ is finitely generated and is left noetherian.
The statement of the theorem now follows from the isomorphism $R^H\cong (K[G]^H\otimes R)^G$ in Lemma~\ref{lemma:transfer}.
\end{proof}

%%%%%%%%%%%%%%%%%%%%%%%%%%%%%%%%%%%%%%%

\end{document}